\documentclass[11pt]{amsart}
\usepackage{amsmath,amsthm,amsfonts,amscd,flafter,epsf}
\usepackage{amssymb,graphicx,color,array}
\usepackage{epsfig,manfnt,mathrsfs}
\usepackage{amsfonts}
\usepackage{amssymb}
\usepackage{amsmath}
\usepackage{arydshln}
\usepackage{tikz}
\usepackage{tikz-cd}
\usetikzlibrary{matrix,calc,arrows}

\usepackage{amsthm}
\usepackage{latexsym}
\usepackage{amscd}
\usepackage{flafter}
\usepackage{epsf}
\usepackage{tikz,tikz-cd}
\usepackage[all]{xy}
\usetikzlibrary{matrix,arrows,decorations.pathmorphing}
\usepackage[a4paper,footskip=2.3cm,headheight=20pt,top=2.3cm,
bottom=2.3cm,right=2.3cm, left=2.3cm]{geometry}

\newcommand{\colvec}[2][.8]{%
  \scalebox{#1}{%
    \renewcommand{\arraystretch}{.8}%
    $\begin{pmatrix}#2\end{pmatrix}$%
  }
}

\newcommand{\ra}{\rightarrow}

\newcommand{\Hbb}{\mathbb{H}}

\newcommand{\C}{\mathbb{C}}

\newcommand{\Z}{\mathbb{Z}}

\newcommand{\Dfrak}{\mathfrak{D}}

\newcommand{\ffrak}{\mathfrak{f}}

\newcommand{\ov}{\overline}

\def\endproof{\relax\ifmmode\expandafter\endproofmath\else
  \unskip\nobreak\hfil\penalty50\hskip.75em\hbox{}\nobreak\hfil\bull
  {\parfillskip=0pt \finalhyphendemerits=0 \bigbreak}\fi}
\def\endproofmath$${\eqno\bull$$\bigbreak}
\def\bull{\vbox{\hrule\hbox{\vrule\kern3pt\vbox{\kern6pt}\kern3pt\vrule}\hrule}}

\newcommand{\Ker}{\mathrm{Ker}}

\newcommand{\Coker}{\mathrm{Coker}}

\newtheorem{thm}{Theorem}[section]
\newtheorem{prop}[thm]{Proposition}

\newtheorem{lem}[thm]{Lemma}

\newtheorem{remark}[thm]{Remark}

\newtheorem{quest}{Question}[section]

\newcommand{\HF}{\mathrm{HF}}

\newcommand{\HFK}{\mathrm{HFK}}
\newcommand{\CFK}{\mathrm{CFK}}
\newcommand{\rank}{\mathrm{rank}}
\newcommand{\Image}{\text{Im}}

\begin{document}
\title[Heegaard Floer homology, degree-one maps]
{Heegaard Floer homology, degree-one maps and splicing 
knot complements}%
\author{Narges Bagherifard, Eaman Eftekhary}%
\address{Department of Mathematical Sciences, Sharif 
University of Technology, Azadi Ave., Tehran, Iran}
\email{n.bagherifard@gmail.com}
\address{School of Mathematics, Institute for Research in 
Fundamental Science (IPM),
P. O. Box 19395-5746, Tehran, Iran}%
\email{eaman@ipm.ir}
\date{January 2018}%
\begin{abstract}
Let $K$ denote a knot inside the homology sphere $Y$ and 
$K'$ denote a knot inside a homology sphere $L$-space. Let 
$X=Y(K,K')$ denote  the 3-manifold obtained by splicing the 
complements of $K$ and $K'$. We show that 
$\rank(\widehat{\mathrm{HF}}(X)) 
\ge \rank(\widehat{\mathrm{HF}}(Y))$.
\end{abstract}
\maketitle

\section{Introduction}
For closed, connected and oriented 3-manifolds $Y_1$ 
and $Y_2$ we say that $Y_1$ $k$-dominate $Y_2$ if there is a map 
$f:Y_1 \rightarrow Y_2$ such that
$f_*:\mathbb{Z}=\mathrm{H}_3(Y_1,\mathbb{Z}) 
\rightarrow \mathrm{H}_3(Y_2,\mathbb{Z})=\mathbb{Z}$ is 
multiplication by $k$. The map $f$ is then called a degree $k$ map. 
$1$-dominance gives a partial ordering of closed oriented 
$3$-manifolds. For every 
topological invariant $\sigma$, the following question is natural.
{\emph{If $Y_1$ $1$-dominates $Y_2$, is 
$\sigma(Y_1)$ at least as large as $\sigma(Y_2)$?}}
In many cases this question has positive answers, 
e.g. $\sigma(Y_1) \ge \sigma(Y_2)$ when $\sigma$ is either the 
rank of $\pi_1$,  or Haken number, 
or Gromov's simplicial volume \cite{Gromov-book}. 
In a sense, $Y_1$ is more complex than $Y_2$ when $Y_1$ 
$1$-dominate $Y_2$. \\

One of the powerful topological invariants for closed 3-manifolds, 
introduced  by Ozsv\'ath and Szab\'o \cite{OS-3m1},  is Heegaard 
Floer homology. Different versions of this homology for a closed 
$3$-manifold $Y$  are denoted by $\widehat{\text{HF}}(Y)$, 
$\text{HF}^+(Y)$, $\text{HF}^-(Y)$, $\text{HF}^\infty(Y)$ and 
$\text{HF}_{\text{red}}(Y)$. It is interesting to investigate if 
there is a positive answer to the above question for Heegaard 
Floer homology. In particular, one would like to investigate the 
following question
\begin{quest} \label{que:1}
Is it true that if $Y_1$ $1$-dominates $Y_2$ then 
$\mathrm{rank}(\widehat{\mathrm{HF}}(Y_1)) 
\ge \mathrm{rank}(\widehat{\mathrm{HF}}(Y_2))$?
\end{quest} 
One attempt is taken by Karakurt and Lidman \cite{Karakurt-Lidman}. 
They prove that if $f :Y_1\rightarrow Y_2$ is a map between 
Seifert fibered 
homology spheres, then
\[\text{rank}(\HF_{\text{red}}(Y_1))\ge\left|\text{deg}(f)\right| 
\text{rank}(\HF_{\text{red}}(Y_2)).\]

Let $X_K$ denote the complement of a tubular neighborhood of a 
knot $K$ in a homology sphere $Y$. Let $\mu_K$ and $\lambda_K$ 
denote the meridian and Seifert longitude of $K$, respectively, 
viewed as curves in $\partial X_K$. For knots $K_1$ and $K_2$ 
in the homology spheres $Y_1$ and $Y_2$ (respectively), let 
$Y=Y(Y_1,Y_2)$ denote the manifold obtained by gluing $X_{K_1}$ 
and $X_{K_2}$ via an orientation-reversing diffeomorphism 
$\phi:\partial X_{K_1} \rightarrow \partial X_{K_2}$ taking 
$\mu_{K_1}$ to $\lambda_{K_2}$ and $\lambda_{K_1}$ to $\mu_{K_2}$. 
We say that $Y=Y(K_1,K_2)$ is obtaind by splicing the knot 
complements $X_{K_1}$ and $X_{K_2}$. The Mayer-Vietoris sequence 
shows that $Y(K_1, K_2)$ is a homology sphere. The image of 
$\partial X_{K_1}$ is incompressible in $Y(K_1, K_2)$ if and only 
if the knots $K_1$ and $K_2$ are both nontrivial. There is a 
natural degree-one map $f_i: Y(K_1, K_2) \rightarrow Y_i$. One can 
restrict attention to this case and ask whether the rank of 
$\widehat\HF(Y(K_1,K_2))$ is greater than or equal to the rank of 
$\widehat\HF(Y_i)$, $i=1,2$.
We prove a number of upper bounds on the rank of 
$\widehat\HF(Y(K_1,K_2))$ in this paper. In particular, we answer the 
aforementioned question positively if $Y_2$ is an $L$-space.

\begin{thm} \label{thm:main}
Let $Y_1$ and $Y_2$ be integral homology spheres  
and $K_i$ denote a knot in $Y_i$ for $i=1,2$. 
If $Y_2$ is an $L$-space then
\begin{equation}\label{eq:main-inequality}
\rank(\widehat{\mathrm{HF}}(Y(K_1,K_2))) 
\ge \rank(\widehat{\mathrm{HF}}(Y_1)).
\end{equation}
\end{thm}

Moreover, we obtain strong restrictions on knot Floer homology of 
$K_1$ and $K_2$ when $Y_1$ and $Y_2$ are arbitrary homology spheres 
and the inequality of Equation~\ref{eq:main-inequality} is 
violated.\\

We use the splicing formula developed in 
\cite{Eftekhary-splicing} and \cite{Eftekhary-incompressible} by 
the second author. This splicing formula is reviewed in 
Section~\ref{sec:splicing}. We then use linear algebra in 
Section~\ref{sec:linear-algebra} 
to obtain restrictions on $K_1$ and $K_2$ when the inequality
of Equation~\ref{eq:main-inequality} is violated.
A number of lemmas are proved in Section~\ref{sec:lemmas} which 
relate the splicing formula to the double filtration on the 
knot Floer complexes associated with $K_1$ and $K_2$.
In Section~\ref{sec:proof} we combine the results of the previous 
sections to prove Theorem~\ref{thm:main}.

\section{Floer Homology and Splicing Knot Complements} 
\label{sec:splicing}
We recall some definitions and theorems from 
\cite{Eftekhary-incompressible}.
Let $K$ be a knot in a homology sphere $Y$. One can associate a 
doubly pointed Heegaard diagram $(\Sigma,\alpha,\beta;u,v)$ to 
$K \subset Y$. The markings $u$ and $v$ can be used to give the map 
\[\mathfrak{s} =\mathfrak{s}_{u,v}:\mathbb{T}_{\alpha}\cap 
\mathbb{T}_{\beta}\rightarrow\mathrm{\underline{Spin}^c}(Y,K)=\Z,\]
as defined in \cite{OS-knots}, where 
$\mathrm{\underline{Spin}^c}(Y,K)$ 
is the set of relative $\mathrm{Spin^c}$ structures for $(Y,K)$. Let 
\[C=C_K=\langle \left[ x,i,j \right] | \, x\in \mathbb{T}_\alpha \cap 
\mathbb{T}_\beta, \, \mathfrak{s}(x)-i+j=0 \rangle _{\mathbb{Z}}\]
denote the $\mathbb{Z} \oplus \mathbb{Z}$ filtered chain complex 
associated with $K$. As in \cite{OS-surgery}, we consider the 
sub-modules
\[ C \{ i=a , j=b \}, C\{ i=a, j\le b \} \quad  \text{and} 
\quad  C\{ i \le a, j=b \} \qquad  a,b \in \mathbb{Z} 
\cup \{\infty \} \]
equipped with the induced structure as a chain complex. Set 
$C \{i = a \}=C \{ i = a, j \le \infty \}$ and 
$C \{ j = b \}=C \{i \le \infty , j = b \}$. Let 
$ \Xi :C\{ i=0 \} \rightarrow C\{j=0 \}$ be the chain homotopy 
equivalence corresponding to the Heegaard moves which change 
the diagram $ (\Sigma ,\alpha ,\beta; u)$ to 
$ (\Sigma ,\alpha ,\beta; v)$. For any relative $\mathrm{ Spin^c}$  
class $\mathfrak{s} \in  \mathbb{Z}  
=\mathrm{ \underline{Spin}^c} (Y, K)$, consider the chain maps
\begin{align*}
& \textit{i}_n^\mathfrak{s}=i_n^\mathfrak{s}(K): C \{i 
\le \mathfrak{s} , j=0 \} \oplus C \{i=0, j \le n-\mathfrak{s}-1 \} 
\rightarrow C \{j=0 \} &\\
& \textit{i}_n^\mathfrak{s}(\left[x,i,0 \right],\left[y,0,j\right])
:=\left[x,i,0\right] + \Xi \left[y,0,j\right].
\end{align*} 
Let $ Y_n(K)$ denote the three-manifold obtained from $Y$ by 
$n$-surgery on $K$ and $K_n$ denote the corresponding knot 
inside $Y_n(K)$. The second author proves the following 
as \cite[Proposition 1.5]{Eftekhary-incompressible}:

\begin{prop}\label{prop:1}
The homology of the mapping cone $M(i_n^\mathfrak{s})$ gives
\[\mathbb{H}_n(K,\mathfrak{s})
=\widehat{\HFK}(Y_n(K),K_n, \mathfrak{s})\]
\end{prop}
Note that $M(i_0^\mathfrak{s})$ is a sub-complex of both 
$M(i_1^\mathfrak{s})$ and $M(i_1^{\mathfrak{s}+1})$ and the quotient 
of $M(i_1^\mathfrak{s})$ by $M(i_0^\mathfrak{s})$  is isomorphic to 
$C\{ i=0, j=-\mathfrak{s}\}\simeq \widehat{\CFK}(K,\mathfrak{s})$.
The quotient of $M(i_1^\mathfrak{s})$ by $M(i_0^{\mathfrak{s}-1})$  
is $C\{i=\mathfrak{s},j=0\}=\widehat{\CFK}(K,\mathfrak{s})$. 
We thus obtain the following two short exact sequences 
\begin{align*}
& 0 \rightarrow M(i_0^\mathfrak{s}) 
\xrightarrow{F_\infty ^\mathfrak{s}} M(i_1^\mathfrak{s}) 
\xrightarrow{F_0 ^\mathfrak{s}} \widehat{\CFK}(K,\mathfrak{s}) 
\rightarrow 0  \ \ \  
\mathrm{and} \ \ \  0 \rightarrow M(i_0^\mathfrak{s-1}) 
\xrightarrow{\bar{F}_\infty ^\mathfrak{s}} M(i_1^\mathfrak{s}) 
\xrightarrow{\bar{F}_0 ^\mathfrak{s}} \widehat{\CFK}(K,\mathfrak{s}) 
\rightarrow 0 &
\end{align*} 
which give the following two homology exact triangles 
\[
\begin{tikzcd}
 & \mathbb{H}_\infty(\mathfrak{s}) 
 \arrow{dr}{\mathfrak{f}_1^{\mathfrak{s}}} \\
\mathbb{H}_1(\mathfrak{s}) \arrow{ur}{\mathfrak{f}_0^{\mathfrak{s}}}  
&& \arrow{ll}{\mathfrak{f}_\infty^{\mathfrak{s}}} 
\mathbb{H}_0(\mathfrak{s})
\end{tikzcd}
\quad \text{and} \quad
\begin{tikzcd}
 & \mathbb{H}_\infty(\mathfrak{s}) 
 \arrow{dr}{\bar{\mathfrak{f}}_1^{\mathfrak{s}}} \\
\mathbb{H}_1(\mathfrak{s}) 
\arrow{ur}{\bar{\mathfrak{f}}_0^{\mathfrak{s}}}  && 
\arrow{ll}{\bar{\mathfrak{f}}_\infty^{\mathfrak{s}}} 
\mathbb{H}_0(\mathfrak{s}-1)
\end{tikzcd}
\]
where 
$\mathbb{H}_\bullet(\mathfrak{s})
=\mathbb{H}_\bullet(K,\mathfrak{s})$. 
Changing the role of the two punctures of 
$(\Sigma,\alpha,\beta;u,v)$,
which corresponds to changing the orientation of $K$, gives the 
duality maps
$$ \tau_\bullet=\tau_\bullet(K): \mathbb{H}_\bullet(K) 
\rightarrow  \mathbb{H}_\bullet(K),\qquad\bullet\in\{0,1,\infty \}.$$
These duality maps take $\mathbb{H}_\bullet(K,\mathfrak{s})$ to 
$\mathbb{H}_\bullet(K,-\mathfrak{s})$ for $\bullet=1, \infty$, and to 
$\mathbb{H}_\bullet(K,1-\mathfrak{s})$ for $\bullet=0$. In  
a basis for $ \mathbb{H}_\bullet(K) $ where $ \mathfrak{f}_\bullet $ 
takes the block form $\colvec{ 0 & 0 \\ I & 0}$, 
the map $\tau_\bullet$ and its inverse $\tau_\bullet^{-1}$ take 
the following matrix block forms
$$ \tau_\bullet=\begin{pmatrix}
  A_\bullet & B_\bullet \\  C_\bullet & D_\bullet
\end{pmatrix}
\qquad \text{and} \qquad \tau_\bullet^{-1}=\begin{pmatrix}
  A_\bullet & B_\bullet \\  \ov{C}_\bullet & D_\bullet
\end{pmatrix}
\qquad\bullet \in \{0, 1, \infty \}.$$
We then have 
$\bar{\mathfrak{f}}_0
=\tau_\infty^{-1}\circ\mathfrak{f}_0\circ\tau_1$, 
$\bar{\mathfrak{f}}_1
=\tau_0^{-1}\circ \mathfrak{f}_1 \circ \tau_\infty$ and 
$\bar{\mathfrak{f}}_\infty=\tau_1^{-1}\circ \mathfrak{f}_\infty
\circ\tau_0$.
Let $a_\bullet=a_\bullet(K)$ denote the rank of 
$\mathfrak{f}_\bullet$ for 
$ \bullet=0, 1, \infty$. Then $a_1$, $a_\infty$ and $a_0+1$ have the 
same parity. Note that $B_0$, $B_1$ and $B_\infty$ are 
matrices of size $ a_\infty \times a_1$, $a_0 \times a_\infty$ and 
$a_1 \times a_0$, respectively.
Define $X_\bullet=X_\bullet(K)$ by $X_0=B_1 B_0 B_\infty$, 
$X_1=B_\infty B_1 B_0$ and $X_\infty=B_0 B_\infty B_1$. By 
\cite[Lemma 5.4]{Eftekhary-incompressible}, the sqaure of 
$X_\bullet$ is zero
for $ \bullet \in \{0, 1, \infty \} $. In particular, if the knot 
$K$ is non-trivial both the kernel and the cokernel of $X_\bullet$ 
are non-trivial.
In the following sections we need to makes several changes of basis 
to obtain suitable forms of the matrices $\tau_\bullet$ to simplify 
the computations. One way to do this is using the matrices 
$P_\bullet$ and $Q_\bullet$, where $P_\bullet$ is an invertible 
$a_\bullet \times a_\bullet$ 
matrix and the matrices $Q_\bullet$ are arbitrary matrices
of correct size. Then one can chooses a change of basis for either of 
$\mathbb{H}_0(K)$, $\mathbb{H}_1(K)$ and $\mathbb{H}_\infty(K)$,
called an {\emph{admissible}}  change of basis,
which is given by the invertible matrices
$$\mathbb{P}_0=
\begin{pmatrix}
 P_\infty & 0 \\
 Q_0      & P_1
\end{pmatrix}, \quad
\mathbb{P}_1=
\begin{pmatrix}
 P_0 & 0 \\
 Q_1      & P_\infty
\end{pmatrix}
\quad \mathrm{and} \quad
\mathbb{P}_\infty=
\begin{pmatrix}
 P_1 & 0 \\
 Q_\infty  & P_0
\end{pmatrix},
$$
respectively. The block forms 
$\ffrak_\bullet=\colvec{0&0\\ I&0}$
remain unchanged under such a change of basis.\\

 Let $Y=Y(K_1,K_2)$ denote the three-manifold 
obtained by splicing the complements of $K_1 \subset Y_1$ and 
$K_2 \subset Y_2$. Let 
$\Box_\bullet ^ \star=\Box_\bullet (K_\star)$, for 
$\Box\in\{ A,B,C,\ov{C},D,X,\tau\}$, $\bullet\in\{0,1,\infty\}$ and 
$\star \in {1,2}$. The following proposition is proved 
in \cite[Subsection 5.3]{Eftekhary-splicing}.

\begin{prop}\label{prop:2}
If $K_i$ is a knot inside the homology sphere $Y_i$ for $i=1,2$,
$$ \mathrm{rank}\; \widehat{\HF}(Y(K_1,K_2) ; 
\mathbb{F})=\mathrm{rank}\left(\mathrm{Ker}(\mathfrak{D})\right)
+\mathrm{rank}\left(\mathrm{Coker}(\mathfrak{D})\right)=:
h(\Dfrak),$$
where the matrix $\mathfrak{D}=\mathfrak{D}(K_1,K_2)$ is given by
\begin{equation}\label{eq:main-matrix}
\colvec[.7]{
D_\infty^1B_1^1\otimes B_1^2A_0^2&B_1^1A_0^1\otimes I&
B_1^1B_0^1\otimes I&D_\infty^1 A_1^1\otimes B_1^2A_0^2&I
\otimes B_1^2B_0^2&0\\ \hdashline
I\otimes B_\infty^2 B_1^2 &D_1^1A_0^1\otimes B_\infty^2 A_1^2 &
D_1^1B_0^1\otimes B_\infty^2 A_1^2& 0&B_0^1B_\infty^1\otimes I&
B_0^1 A_\infty^1\otimes I\\ \hdashline
I\otimes D_\infty^2 B_1^2& \begin{array}{c}I\otimes I+\\
D_1^1 A_0^1\otimes D_\infty^2 A_1^2\end{array}& 
D_1^1 B_0^1\otimes D_\infty^2 A_1^2 
&0&0&0\\ \hdashline
B_\infty^1 B_1^1\otimes I&0&I\otimes B_0^2B_\infty^2& 
B_\infty^1 A_1^1\otimes I
&\begin{array}{c} D_0^1 B_\infty^1\otimes B_0^2A_\infty^2\\ 
+X_1^1B_\infty^1\otimes B_0^2X_1^2\end{array}
&\begin{array}{c} D_0^1 A_\infty^1\otimes B_0^2A_\infty^2\\ 
+X_1^1A_\infty^1\otimes B_0^2X_1^2\end{array}\\ \hdashline
D_\infty^1 B_1^1\otimes D_1^2A_0^2&0&0&\begin{array}{c}I\otimes I+\\
D_\infty^1 A_1^1\otimes D_1^2A_0^2\end{array}
&I\otimes D_1^2B_0^2&0\\ \hdashline
0&0&I\otimes D_0^2B_\infty^2 &0&\begin{array}{c} 
D_0^1 B_\infty^1\otimes D_0^2A_\infty^2\\
+X_1^1B_\infty^1\otimes D_0^2X_1^2\end{array}&\begin{array}{c}
I\otimes I+\\ D_0^1A_\infty^1\otimes D_0^2 A_\infty^2\\
+X_1^1 A_\infty^1\otimes D_0^2X_1^2\end{array}
}.
\end{equation}
\end{prop}
We call the matrices $M$ and $N$ {\emph{equivalent}} if the rank of 
$\Ker(M)$ is equal to the rank of $\Ker(N)$ and the rank of 
$\Coker(M)$ is equal to the rank of $\Coker(N)$. In our path towards 
proving Theorem~\ref{thm:main}, we try to change the matrix $\Dfrak$
to simpler equivalent matrices.

\section{Splicing formula and linear algebra} 
\label{sec:linear-algebra}
\begin{prop}\label{prop:3}
With our previous notation fixed, if 
$\rank\left(\widehat{\HF}(Y(K_1,K_2))\right)
<\rank\left(\widehat{\HF}(Y_1)\right)$,
then one of the following $5$ conditions is satisfied, where 
$r_i^2=\rank(B_i^2)$.
\begin{align*}
& \mathbf{(S_1)}\  r_0^2 \le r_1^2 = r_\infty^2 = a_1^2 
= a_\infty ^2 < a_0^2,\quad\quad\quad\quad\ \quad
\mathbf{(S_4)} \ r_0^2 = a_\infty^2, r_\infty^2 = a_0^2 \ \ 
\text{and} \ \  a_1^2 \geq a_0^2 , a_\infty^2,   \\
& \mathbf{(S_2)} \ r_0^2 = r_1^2 = a_\infty^2 \le r_\infty^2 \ \ 
\text{and} \ \  a_\infty^2 <  a_1^2 , a_0^2, 
\quad \quad \mathbf{(S_5)} \ r_0^2 = a_1^2, r_1^2 = a_0^2 \ \  
\text{and} \ \  a_\infty^2 \geq  a_0^2, a_1^2,\\
& \mathbf{(S_3)} \ r_0^2 = r_\infty^2 = a_1^2 \le r_1^2 \ \ 
\text{and} \ \  a_1^2 < a_\infty^2 , a_0^2.
\end{align*}
\end{prop}
\begin{proof}
Let us assume that $B_\infty ^i z_1 ^i= B_1 ^i z_0 ^i= 
B_0 ^i z_\infty ^i= 0$ and
\begin{align*}
\colvec{
B_1 ^i & 0 \\
 D_1 ^i+A_0 ^i & B_0 ^i
}
\colvec{
 x_\infty ^i \\
 y_\infty ^i
}
=
\colvec{
 B_0 ^i & 0 \\
 D_0 ^i+A_\infty ^i & B_\infty ^i
}
\colvec{
 x_1 ^i \\
 y_1 ^i
}
=\colvec{
 B_\infty ^i & 0 \\
 D_\infty ^i+A_1 ^i & B_1 ^i
}
\colvec{
 x_0 ^i \\
 y_0 ^i
}
=0.
\end{align*}
Then it can easily be checked that
\[\mathfrak{D}
\colvec{
 y_0 ^1 \otimes x_\infty ^2 + x_\infty ^1 \otimes y_0 ^2 
 + z_0 ^1 \otimes z_0 ^2 \\
 x_\infty ^1 \otimes x_0 ^2 \\
 y_\infty ^1 \otimes x_0 ^2 + x_1 ^1 \otimes y_1 ^2 
 + z_\infty ^1 \otimes z_1 ^2 \\
 x_0 ^1 \otimes x_\infty ^2 \\
 y_1 ^1 \otimes x_1 ^2 + x_0 ^1 \otimes y_\infty ^2 
 + z_1 ^1 \otimes z_\infty ^2 \\
 x_1 ^1 \otimes x_1 ^2
}=0.\]
Elements in 
$ \Ker( \mathfrak{f}_\bullet ^i 
+ \overline{\mathfrak{f}}_\bullet ^i )$ and 
$K_\bullet ^i=\Ker(\mathfrak{f}_\bullet ^i 
\cap \overline{\mathfrak{f}}_\bullet ^i )$ 
correspond to the vectors of form
$\colvec{ x_\bullet ^i \\ y_\bullet ^i}$ 
and 
$\colvec{ 0 \\ z_\bullet ^i}$, respectively. Let us set
\[L_\bullet ^i=
\Ker( \mathfrak{f}_\bullet ^i + \overline{\mathfrak{f}}_\bullet ^i)
/\big(\Ker (\mathfrak{f}_\bullet ^i) \cap\Ker(  
\overline{\mathfrak{f}}_\bullet ^i )\big),\quad
k_\bullet ^i=\text{dim} (K_\bullet ^i) \quad \text{and} 
\quad  l_\bullet ^i=\text{dim} (L_\bullet ^i)
\]
for $i=1,2$ and  $\bullet \in \{0, 1, \infty \}$.
We then find
\[\text{dim} ( \Ker(\mathfrak{D})) \ge 
k_0^1 k_0^2 + k_\infty^1 k_1^2 + k_1^1 k_\infty^2 
+ l_\infty^1 l_0^2 + l_0^1 l_\infty^2 + l_1^1 l_1^2 .\]
Similarly, we have
\begin{align*}
&M_\bullet ^i:= 
\Coker( \mathfrak{f}_\bullet ^i, \bar{\mathfrak{f}}_\bullet ^i),  
\quad  N_\bullet ^i:=
\frac{\Image( \mathfrak{f}_\bullet ^i)
+\Image( \bar{\mathfrak{f}}_\bullet ^i)}
{\Image (\mathfrak{f}_\bullet ^i +  
\bar{\mathfrak{f}}_\bullet ^i )},\quad
c_\bullet ^i:=\text{dim}(M_\bullet ^i) \quad  
\text{and} \quad 
d_\bullet ^i:=\text{dim}(N_\bullet ^i),\\
\Rightarrow\quad&
\text{dim} ( \text{Coker} (\mathfrak{D})) \ge 
c_\infty^1 c_\infty ^2 + c_0^1 c_1 ^2 + c_1^1 c_0 ^2 
+ d_\infty^1 d_0 ^2 + d_0^1 d_\infty ^2 + d_1^1 d_1 ^2.
\end{align*}
 Denote the rank of $B_\bullet ^i$ by $r_\bullet ^i$. 
 Consequently, we find
\begin{align*}
& k_0^i =a_\infty^i - r_1 ^i,   
&& l_0^i =a_0^i - r_\infty ^i - \delta_0 ^i, 
&& c_0^i =a_1^i - r_\infty ^i, 
&&d_0^i =a_0^i - r_1 ^i - \delta_0 ^i, 
\\
&k_1^i =a_0^i - r_\infty ^i,
&&l_1^i =a_1^i - r_0 ^i - \delta_1 ^i,  
&&c_1^i =a_\infty^i - r_0 ^i ,
&&d_1^i =a_1^i - r_\infty ^i - \delta_1 ^i,
\\
&k_\infty^i =a_1^i - r_0 ^i,
&&l_\infty^i =a_\infty^i - r_1^i - \delta_\infty^i,
&&c_\infty^i =a_0^i - r_1 ^i,  
&&d_\infty^i =a_\infty^i - r_0 ^i - \delta_\infty ^i,
\end{align*}
where we have
\begin{align*}
& 0 \le r_0^i \le \min\{ a_1^i, a_\infty^i \}, 
&&0 \le r_1^i \le \min\{ a_\infty^i, a_0^i \},
&&0 \le r_\infty^i \le \min\{ a_0^i, a_1^i \},
\\
&0 \le \delta_0^i\leq a_0^i - \max\{ r_1^i, r_\infty^i \}, 
&&0 \le \delta_1^i\leq a_1^i - \max\{ r_\infty^i, r_0^i \},
&&  0 \le \delta_\infty^i\leq a_\infty^i - \max\{ r_0^i, r_1^i \}. 
\end{align*}
Let us denote the rank of $\widehat{\HF}((Y_i)_\bullet(K_i))$ by 
$y_\bullet^i$. We would like to make sure that
\begin{align*}
h(\mathfrak{D})&\geq  y_\infty^1= k_\infty^1+ l_\infty^1 
+ c_\infty^1+d_\infty^1= 2 a_\infty ^1 + a_1^1 + a_0^1 - 
2 r_0^1 - 2 r_1^1 - 2 \delta_\infty ^1. 
\end{align*}
From the above considerations we find
\begin{align*}
h(\mathfrak{D})-y_\infty^1 &
\geq ( a_\infty ^1 - r_1^1 ) ( a_\infty ^2 - r_1^2 ) 
+ ( a_1 ^1 - r_0^1 ) ( a_0 ^2 - r_\infty^2 -1) 
+ ( a_0 ^1 - r_\infty^1 ) ( a_1 ^2 - r_0^2 )  \\
  &+ ( a_0 ^1 - r_1^1 ) ( a_0 ^2 - r_1^2-1 ) 
  + ( a_\infty ^1 - r_0^1 ) ( a_1 ^2 - r_\infty^2 ) 
  + ( a_1 ^1 - r_\infty^1 ) ( a_\infty ^2 - r_0^2 )  \\
&+(a_\infty^1-\delta_\infty^1-r_1^1)(a_0^2-\delta_0^2-r_\infty^2-1) 
+(a_\infty^1-\delta_\infty^1-r_0^1)(a_0^2-\delta_0^2 - r_1^2 -1)\\
&+ (a_0^1-\delta_0^1 - r_\infty^1)(a_\infty^2-\delta_\infty^2-r_1^2) 
+ (a_0^1 - \delta_0^1 - r_1^1)( a_\infty ^2-\delta_\infty^2-r_0^2)\\
  &+ (a_1^1 - \delta_1^1 - r_0^1) ( a_1 ^2 - \delta_1^2 - r_0^2 ) 
  + (a_1^1 - \delta_1^1 - r_\infty^1)(a_1^2-\delta_1^2-r_\infty^2 )
\\
  &= \delta_\infty^1 (a_1^2 + a_\infty^2 - r_1^2 - r_\infty^2) 
  + \delta_1^1 (a_0^2 + a_\infty^2 - r_0^2 - r_\infty^2 -1) 
  + \delta_0^1 (a_0^2 + a_1^2 - r_0^2 - r_1^2 -1)  \\
  &+  (a_\infty^1 - \delta_\infty^1 - r_1^1) 
  (a_\infty^2 + a_0^2 - \delta_0^2 - r_\infty^2 - r_1^2 -1)  \\
  &+  (a_\infty^1 - \delta_\infty^1 - r_0^1) 
  (a_0^2 + a_1^2 - \delta_0^2 - r_1^2 - r_\infty^2 -1)  \\
  &+  (a_0^1 - \delta_0^1 - r_\infty^1) 
  (a_1^2 + a_\infty^2 - \delta_\infty^2 - r_0^2 - r_1^2 )  
  +  (a_0^1 - \delta_0^1 - r_1^1) 
  (a_0^2 + a_\infty^2 - \delta_\infty^2 - r_0^2 - r_1^2 )  \\
  &+  (a_1^1 - \delta_1^1 - r_0^1) 
  (a_0^2 + a_1^2 - \delta_1^2 - r_0^2 - r_\infty^2 )  
  +  (a_1^1 - \delta_1^1 - r_\infty^1) 
  (  a_1^2 + a_\infty^2 - \delta_1^2 - r_0^2 - r_\infty^2 ) . 
\end{align*}
The expressions in parentheses  on the left-hand-side of the above 
equation are non-negative except for the following $4$ expressions, 
which are greater than or equal to $-1$:
\begin{align*}
&a_0^2 + a_\infty^2 - r_0^2 - r_\infty^2 -1,&&
&&a_0^2 + a_1^2 - r_0^2 - r_1^2 -1,\\
&a_\infty^2 + a_0^2 - \delta_0^2 - r_\infty^2 - r_1^2 -1&&\text{and}
&&a_0^2 + a_1^2 - \delta_0^2 - r_1^2 - r_\infty^2 -1.
\end{align*}
If the first expression is equal to $-1$, it follows that 
$r_1^2\leq a_0^2=r_\infty^2\leq a_1^2$ and 
$r_1^2\leq a_\infty^2=r_0^2\leq a_1^2$. We are thus 
in case $\textbf{S}_4$. 
Similarly, if the second expression above is equal to 
$-1$, it follows that we are in case $\textbf{S}_5$. Away from these 
two cases,  it follows from 
$h(\mathfrak{D})<y_\infty^1$ that the third or fourth 
expression above is equal to $-1$. 
Thus, $a_\infty^2=\min\{r_1^2,r_\infty^2\}$
or $a_1^2=\min\{r_1^2,r_\infty^2\}$.\\

\textbf{ (Case I) } $ a_\infty^2= \min\{ r_1^2 , r_\infty ^2 \} $. 
In this case we have $ r_0^2 \le r_1^2 = a_\infty^2 \le r_\infty^2$.
From the assumption it is clear that 
$a_\infty^2 \le r_\infty^2 \le \text{min}\{a_0^2,a_1^2\}$. 
In particular $a_0^2 > a_\infty^2$ and $r_0^2 < a_0^2$. If 
$a_\infty^2=a_1^2$, we have case $\textbf{S}_1$. 
Otherwise $a_1^2 \ge a_\infty^2+2$. 
If we have $a_0^2+a_\infty^2-r_0^2-r_\infty^2=0$
we are in case $\textbf{S}_4$. Finally, if 
$a_1^2+a_\infty^2-\delta_1^2-r_0^2-r_\infty^2=0$
we are in case $\textbf{S}_2$. 
Away from these three special cases we find
\begin{align*}
& a_1^2+a_\infty^2-r_1^2-r_\infty^2 \ge 0, 
&& a_0^2+a_\infty^2-r_0^2-r_\infty^2\ge 1, 
&& a_0^2+a_1^2-r_0^2-r_1^2\ge 3,\\
& a_\infty^2+a_0^2-\delta_0^2-r_\infty^2-r_1^2=0, 
&& a_0^2+a_1^2-\delta_0^2-r_1^2-r_\infty^2\geq 2,
&&a_1^2+a_\infty^2-\delta_\infty^2-r_1^2-r_0^2 \ge 2,\\
&a_0^2+a_\infty^2-\delta_\infty^2-r_0^2-r_1^2 \ge 1,
&&a_0^2+a_1^2-\delta_1^2-r_\infty^2-r_0^2 \ge 1,
&&a_1^2+a_\infty^2-\delta_1^2-r_\infty^2-r_0^2 \ge 1.
\end{align*}
With these values in place, we compute
\begin{align*}
h(\mathfrak{D})-y_\infty^1 & \ge 2\delta_0^1
-(a_\infty^1-\delta_\infty^1-r_1^1)
+(a_\infty^1-\delta_\infty^1-r_0^1)
+2(a_0^1-\delta_0^1-r_\infty^1) \\
&+(a_0^1-\delta_0^1-r_1^1)+(a_1^1-\delta_1^1-r_0^1)
+(a_1^1-\delta_1^1-r_\infty^1) \\
&=2(a_0^1-r_\infty^1)+(a_0^1-\delta_0^1-r_\infty^1)
+2(a_1^1-\delta_1^1-r_0^1) \ge 0.
\end{align*}

\textbf{ (Case II) } $ a_1^2= \min \{ r_1^2 , r_\infty ^2 \} $. 
In this case we have $ r_0^2 \le r_\infty^2 = a_1^2 \le r_1^2$.
From the assumption it is clear that 
$a_1^2 \le r_1^2 \le \text{min}\{a_0^2,a_\infty^2\}$. In particular, 
$a_0^2 > a_\infty^2$. Thus $r_0^2 < a_0^2$. 
If $a_\infty^2=a_1^2$, we have case $\textbf{S}_1$. Otherwise, 
$a_\infty^2 \ge a_1^2+2$. If $a_0^2+a_1^2-r_0^2-r_1^2=0$
we are in  case $\textbf{S}_5$. Finally, if 
$a_1^2+a_\infty^2-\delta_\infty^2-r_0^2-r_1^2=0$
we are in  case $\textbf{S}_3$. Away from these
three  special cases we find
\begin{align*}
& a_1^2+a_\infty^2-r_1^2-r_\infty^2 \ge 0, 
&&a_0^2+a_\infty^2-r_0^2-r_\infty^2\ge 3,
&&a_0^2+a_1^2-r_0^2-r_1^2\ge 1,
\\
&a_0^2+a_\infty^2-\delta_0^2-r_\infty^2-r_1^2\ge 2, 
&&a_0^2+a_1^2-\delta_0^2-r_1^2-r_\infty^2=0,
&&a_1^2+a_\infty^2-\delta_\infty^2-r_1^2-r_0^2 \ge 1, 
\\
&a_0^2+a_\infty^2-\delta_\infty^2-r_1^2-r_0^2 \ge 1, 
&&a_0^2+a_1^2-\delta_1^2-r_\infty^2-r_0^2 \ge 1,
&&a_1^2+a_\infty^2-\delta_1^2-r_\infty^2-r_0^2 \ge 2.
\end{align*}
With these values in place, we compute
\begin{align*}
h(\mathfrak{D})-y_\infty^1 & \ge 2\delta_1^1
+(a_\infty^1-\delta_\infty^1-r_1^1)-
(a_\infty^1-\delta_\infty^1-r_0^1)
+(a_0^1-\delta_0^1-r_\infty^1) \\
&+(a_0^1-\delta_0^1-r_1^1)+(a_1^1-\delta_1^1-r_0^1)
+2(a_1^1-\delta_1^1-r_\infty^1) \\
&=2(a_1^1-r_\infty^1)+(a_1^1-\delta_1^1-r_\infty^1)
+2(a_0^1-\delta_0^1-r_1^1) \ge 0.
\end{align*}
Since cases I and II are sorted out, the proof is complete.
\end{proof}
\begin{lem}\label{lem:19}
Define $\imath:\{0,1,\infty\}\ra \{0,1,\infty\}$ by 
$\imath(0)=\infty,\imath(1)=1$ and $\imath(\infty)=0$.
Given the pair of knots $(K_1,K_2)$, and 
$(\circ,\bullet,*) \in \{ (0,1,\infty),(1,\infty,0),(\infty,0,1)\}$,
\begin{enumerate}
\item If $B_\circ^2$ is injective and $B_\bullet^2$ is surjective then
 \begin{enumerate}
 \item $\Ker(\mathfrak{D})$ contains a subspace isomorphic to 
 $\Ker(B_{\imath(\circ)}^1B_{\imath(\bullet)}^1) \otimes 
 \Ker(B_\bullet^2 B_\circ^2)$.
 \item $\Coker(\mathfrak{D})$ contains a subspace isomorphic to 
 $\Coker(B_{\imath(\circ)}^1B_{\imath(\bullet)}^1) 
 \otimes \Coker(B_\bullet^2 B_\circ^2)$.
 \end{enumerate}
\item If $B_\circ^2$ is surjective and $B_\bullet^2$ is injective then
 \begin{enumerate}
 \item $\Ker(\mathfrak{D})$ contains a subspace isomorphic to 
 $\Ker(B_{\imath(\bullet)}^1B_{\imath(*)}^1) 
 \otimes \Ker(B_*^2 B_\bullet^2)$.
 \item $\Coker(\mathfrak{D})$ contains a subspace isomorphic to 
 $\Coker(B_{\imath(*)}^1 B_{\imath(\circ)}^1) 
 \otimes \Coker(B_\circ^2 B_*^2)$.
 \end{enumerate}
\end{enumerate}
\end{lem}
\begin{proof}
Let us assume $(\circ,\bullet,*)=(\infty,0,1)$. After an admissible 
change of basis we have
\begin{equation*}
\renewcommand{\arraystretch}{0.8}%
{\scriptsize
\tau_0^2=\left(
\begin{array}{c@{\hspace{2bp}}c@{\hspace{2bp}};{2pt/2pt}c}
0 & 0 & I \\ 
0 & J & 0 \\  \hdashline[2pt/2pt]
I & 0 & 0 
\end{array}
\right)
}
\quad , \quad
\renewcommand{\arraystretch}{0.8}%
{\scriptsize
\tau_\infty^2=\left(
\begin{array}{c@{\hspace{2bp}};{2pt/2pt}c@{\hspace{2bp}}
c}
0 &  0 & I \\ \hdashline[2pt/2pt]
0 & K& 0 \\ 
0 & 0 & 0 
\end{array}
\right)
}
\quad , \quad
\renewcommand{\arraystretch}{0.8}%
{\scriptsize
\tau_1^2=\left(
\begin{array}{c@{\hspace{2bp}}c@{\hspace{2bp}};
{2pt/2pt}c@{\hspace{2bp}}c}
x & y  & A & B  \\ 
z& w  & X_1^2 & D \\   \hdashline[2pt/2pt]
* & * & \alpha & \beta  \\ 
* & * & \gamma  & \delta \\ 
\end{array}
\right).
}
\end{equation*}
If we replace the above block presentations, the matrix
$\Dfrak=\Dfrak(K_1,K_2)$ takes the form
\begin{align*}
\colvec[0.9]{
0 & \star & \star & 0 & \star & 0 & 0 & \star & \star &0\\
0 & D_\infty^1B_1^1\otimes DJ & 0 & B_1^1A_0^1\otimes I & 0 & 
B_1^1B_0^1\otimes I & 0 & D_\infty^1A_1^1\otimes DJ & I\otimes X_1^2 
&0\\
I\otimes X_1^2& \star & \star & \star & \star & \star & 0 & 0 & \star &\star\\
I\otimes KA& \star & \star & \star & \star & \star & 0 & 0 & 0 &0\\
0 & 0 & 0 & I\otimes I & 0 & 0 & 0 & 0 & 0 &0\\
B_\infty^1B_1^1\otimes I& 0 & 0 & 0 & 0 & I\otimes I & \star & 0 & \star 
&\star\\
0 & \star & 0 & 0& 0& 0 & 0 & \star & 0 &0\\
0& \star & 0 & 0 & 0 & 0 & I\otimes I & \star & \star &0\\
0& \star & 0 & 0 & 0 & 0 & 0& \star & \star &0\\
0&0&0&0&0&0&0&0&0&I\otimes I
}
\end{align*}
The identity matrices which appear as $(5,4)$ entry and the $(10,10)$
entry of the above matrix may be used foe cancellations, which change 
$\Dfrak$ to the following equivalent matrix
 \begin{align*}
\colvec[0.9]{
0 & \star & \star & \star & 0 & 0 & \star & \star \\
0 & D_\infty^1B_1^1\otimes DJ & 0 &  0 & 
B_1^1B_0^1\otimes I & 0 & D_\infty^1A_1^1\otimes DJ & I\otimes X_1^2 \\
I\otimes X_1^2& \star & \star & \star & \star & 0 & 0 & \star \\
I\otimes KA& \star & \star  & \star & \star & 0 & 0 & 0 \\
B_\infty^1B_1^1\otimes I& 0 & 0  & 0 & I\otimes I & \star & 0 & \star \\
0 & \star & 0 & 0& 0 & 0 & \star & 0 \\
0& \star & 0 &  0 & 0 & I\otimes I & \star & \star \\
0& \star & 0 & 0  & 0 & 0& \star & \star
}
\end{align*}
Since 
$B_1^2 B_0^2=\colvec{ A\\ X_1^2}$ and 
$B_\infty^2 B_1^2=\colvec{ X_1^2 &  D}$,
corresponding to the first column and the second row of the 
above matrix we obtain the subspace   
$K\simeq\Ker(B_\infty^1B_1^1)\otimes\Ker(B_1^2B_0^2)$ of 
$\Ker(\Dfrak)$ and the subspace $C\simeq 
\Coker(B_1^1B_0^1)\otimes \Coker(B_\infty^2B_1^2)$
of $\Coker(\Dfrak)$. This proves $(2-a)$ and $(2-b)$ when 
$(\circ,\bullet,*)=(\infty,0,1)$. The proof of the other claims 
in Lemma~\ref{lem:19} (which are not used in this paper) is 
completely similar.
\end{proof}
\begin{remark}\label{remark}
If $B_0^2$ is injective and $B_\infty^2$ is surjective,
the subspaces $K\subset\Ker(\Dfrak)$ and $C\subset \Coker(\Dfrak)$
constructed in Lemma~\ref{lem:19} correspond
to the first column and the first row in the block decomposition
of Equation~\ref{eq:main-matrix}.  
Combining with the discussion after 
\cite[Definition 5.2]{Eftekhary-incompressible}, it follows that
 $\Ker(\Dfrak)$ and $\Coker(\Dfrak)$ include 
subspaces isomorphic to 
\begin{align*}
&\left(\Ker(B_0^1)\otimes \Ker(B_\infty^2)\right)\oplus 
\left(\Ker(B_\infty^1)\otimes \Ker(B_0^2)\right)\oplus 
\left(K+K'\right)\quad \text{and}\\ 
&\left(\Coker(B_0^1)\otimes \Coker(B_\infty^2)\right)\oplus 
\left(\Coker(B_\infty^1)\otimes \Coker(B_0^2)\right)\oplus 
\left(C+C'\right).
\end{align*} 
where $K,K'\subset \Ker(\ffrak_0^1)\otimes\Ker(\ffrak_0^2)$ 
are subspaces isomorphic to 
$\Ker(B_\infty^1B_1^1)\otimes \Ker(B_1^2B_0^2)$ and 
 $\Ker(B_1^1)\otimes\Ker(B_1^2)$ (respectively), and
$C,C'\subset \Coker(\ffrak_\infty^1)\otimes \Coker(\ffrak_\infty^2)$
 are subspaces isomorphic to 
 $\Coker(B_1^1B_0^1)\otimes \Coker(B_\infty^2B_1^2)$ and 
$\Coker(B_1^1)\otimes\Coker(B_2^2)$ (respectively).
\end{remark}

\section{Splicing and the double filtration of knot Floer complex} 
\label{sec:lemmas}
Let us assume that $C=C_{K}={\CFK}^\infty(K)$ is the
$\Z\oplus\Z$ filtered chain complex associated with the oriented 
knot $K$ inside the homology sphere $Y$.
Let $\mathrm{F}_\mathfrak{s}$ denote the homology of $C\{i \le \mathfrak{s}, j=0 \}$ and
 $\mathrm{F}^\prime_\mathfrak{s}$ denote the homology of 
$C_{-K}\{i \le \mathfrak{s}, j=0 \}$, where $-K$ denotes the knot $K$ with 
the reversed orientation. There are induced maps
$\iota_\mathfrak{s}: \mathrm{F}_\mathfrak{s} \rightarrow \widehat{\HF}(Y)$ and 
$\iota^\prime_\mathfrak{s}: \mathrm{F}^\prime_\mathfrak{s} \rightarrow \widehat{\HF}(Y)$
and we let
$$
\mathrm{H}_{\mathfrak{s},*}=\mathrm{Im}(\iota_\mathfrak{s}), 
\quad \mathrm{K}_{\mathfrak{s},*}=\mathrm{Ker}(\iota_\mathfrak{s}),
\quad \mathrm{H}_{*,\mathfrak{s}}=\mathrm{Im}(\iota^\prime_\mathfrak{s}) \quad \text{and} 
\quad \mathrm{K}_{*,\mathfrak{s}}=\mathrm{Ker}(\iota^\prime_\mathfrak{s}).
$$
There is also an induced map from $K_{\mathfrak{s},*}$ to $K_{\mathfrak{s}+i,*}$ if 
$i>0$. Denote the kernel of this map by 
$[K_{\mathfrak{s},*}]_{\mathfrak{s}+i}$ and denote its image by $[K_{\mathfrak{s},*}]^{\mathfrak{s}+i}$.
Define $[K_{*,\mathfrak{s}}]_{\mathfrak{s}+i}$ and $[K_{*,\mathfrak{s}}]^{\mathfrak{s}+i}$ similarly.
Then
\begin{align*}
\widehat{\HF}(Y)\simeq \bigoplus_{p,q}\mathrm{A}_{p,q}
\simeq \bigoplus_s\mathrm{E}_\mathfrak{s},\quad
\mathrm{A}_{p,q}=\frac{\mathrm{H}_{p,q}}{\mathrm{H}_{p-1,q}
+\mathrm{H}_{p,q-1}}\quad\text{and}\quad
\mathrm{E}_\mathfrak{s}
=\frac{\bigoplus_{p+q=\mathfrak{s}}H_{p,q}}{\bigoplus_{p+q=\mathfrak{s}-1}H_{p,q}}
\simeq \bigoplus_{p+q=\mathfrak{s}}\mathrm{A}_{p,q}.
\end{align*}
where $\mathrm{H}_{\mathfrak{s},\mathfrak{t}}=\mathrm{H}_{\mathfrak{s},*} \cap \mathrm{H}_{*,\mathfrak{t}}$. 
Note that 
$\mathrm{H}_{\mathfrak{s},\mathfrak{t}} \cong \bigoplus_{p\leq \mathfrak{s},q\leq \mathfrak{t}}\mathrm{A}_{p,q}$. 
Denote the rank of $\mathrm{E}_\mathfrak{s}=\mathrm{E}_\mathfrak{s}(K)$ by 
$e_\mathfrak{s}=e_\mathfrak{s}(K)$.
The following lemma is then a direct consequence of 
Proposition~\ref{prop:1}.
\begin{lem}\label{lem:3.1}
With the above notation fixed,
\begin{equation}\label{eq:H-decomposition}
\mathbb{H}_n(\mathfrak{s}) \cong \mathrm{K}_{\mathfrak{s},*} \oplus 
\mathrm{K}_{*,n-\mathfrak{s}-1} \oplus (\bigoplus_{p \le \mathfrak{s}<n-q} 
\mathrm{A}_{p,q}) 
 \oplus (\bigoplus_{ p > \mathfrak{s}\geq n-q} \mathrm{A}_{p,q}).
\end{equation}
\end{lem}

Let us investigate the behaviour of $\mathfrak{f}_\bullet^\mathfrak{s}$ and 
$\bar{\mathfrak{f}}_\bullet^\mathfrak{s}$, $\bullet \in \{0,1,\infty \}$, in 
terms of the above description of the modules $\Hbb_\bullet(K)$. 
The map 
$\mathfrak{f}_\infty^\mathfrak{s}: \mathbb{H}_0(\mathfrak{s}) \rightarrow \mathbb{H}_1(\mathfrak{s})$  
is induced by the inclusion and sends a quadruple
$(x,y,z,w)$, which belongs to the module on the right-hand-side
of Equation~\ref{eq:H-decomposition}, to  $(x,[y],z,[w])$.
Here $[y]$ denotes  the image of $y$ in $\mathrm{K}_{*,-\mathfrak{s}}$ and 
$[w]$ denotes the image of $w$ in 
$\bigoplus_{p>\mathfrak{s}>-q } \mathrm{A}_{p,q}$.
 Similarly, $\bar{\mathfrak{f}}_\infty^{\mathfrak{s}+1}: 
 \mathbb{H}_0(\mathfrak{s}) \rightarrow \mathbb{H}_1(\mathfrak{s}+1)$ sends 
 $(x,y,z,w)$ to $([x],y,z,[w])$, where  $[x]$ denotes 
 the image of $x$ in $\mathrm{K}_{\mathfrak{s}+1,*}$ and $[w]$ denotes 
 the image of $w$ in 
 $\bigoplus_{p-1>\mathfrak{s}\geq -q }  \mathrm{A}_{p,q}$. 
Thus, the following lemma follows.

\begin{lem}\label{lem:3.2}
With the above notation fixed,
\begin{itemize}
\item[1.]
$\Ker({\mathfrak{f}}_\infty^\mathfrak{s})= [\mathrm{K}_{*,-\mathfrak{s}-1}]_{-\mathfrak{s}} \oplus 
\Big(\bigoplus_{  p>\mathfrak{s}=-q } 
\mathrm{A}_{p,q}\Big).$
\item[2.]
$\mathrm{Im}({\mathfrak{f}}_\infty^\mathfrak{s})= \mathrm{K}_{\mathfrak{s},*} \oplus 
[\mathrm{K}_{*,-\mathfrak{s}-1}]^{-\mathfrak{s}} \oplus \Big(\bigoplus_{p \le \mathfrak{s}<-q } 
\mathrm{A}_{p,q}\Big)\oplus\Big(\bigoplus_{ p>\mathfrak{s}>-q}
\mathrm{A}_{p,q}\Big)$
\end{itemize}
\end{lem}
Using Lemma~\ref{lem:3.2}, we can quickly prove the following lemma.
\begin{lem}\label{lem:3.3} Withe the above notation fixed, we have
$\Ker(B_0) \cong \mathrm{E}_1$,
$\Coker(B_1) \cong \mathrm{E}_0$ and 
\begin{align*}
\Ker(B_1) &\cong \bigoplus_{\mathfrak{s}} [\mathrm{K}_{\mathfrak{s},*}]^{\mathfrak{s}+1} \oplus 
\bigoplus_{\mathfrak{s}} [\mathrm{K}_{*,\mathfrak{s}}]^{\mathfrak{s}+1} \oplus \bigoplus_{\mathfrak{s}} 
\mathrm{E}_{\mathfrak{s}}^{\mathrm{max}\{0,|\mathfrak{s}|-1\}}\quad\text{and}\\
\Coker(B_0) &\cong \bigoplus_{\mathfrak{s}} [\mathrm{K}_{\mathfrak{s},*}]^{\mathfrak{s}+1} \oplus 
\bigoplus_{\mathfrak{s}} [\mathrm{K}_{*,\mathfrak{s}}]^{\mathfrak{s}+1} \oplus \bigoplus_{\mathfrak{s}} 
\mathrm{E}_{\mathfrak{s}}^{\mathrm{max}\{0,-\mathfrak{s}\}}.
\end{align*}
\end{lem}
\begin{proof}
Note that $\Ker(B_0) \cong \Ker(\mathfrak{f}_\infty) \cap 
\Ker(\bar{\mathfrak{f}}_\infty)$. To see this note that since 
$\mathfrak{f}_\infty$ takes the block form 
$\left( \begin{smallmatrix}0&0\\ I&0 \end{smallmatrix} \right)$ and 
$\bar{\mathfrak{f}}_\infty
=\tau_1^{-1}\circ\mathfrak{f}_\infty\circ \tau_0$. 
Thus $\left(\begin{smallmatrix}x \\ y\end{smallmatrix}\right)\in 
\Ker (\mathfrak{f}_\infty)\cap\Ker(\bar{\mathfrak{f}}_\infty)$ 
if and only if 
\begin{align*}
\colvec{ 0 & 0 \\ I & 0 } \colvec{ x \\ y} =0 \quad\text{and}\quad
\colvec{A_1&B_1\\ \ov{C}_1&D_1 }\colvec{0&0\\ I&0}
\colvec{A_0&B_0\\C_0& D_0}\colvec{ x \\ y}=0.
\end{align*}
Thus $x=0$ and $B_0 y=0$, proving our claim. 
If $(x,y,z,w) \in \Ker(\mathfrak{f}^\mathfrak{s}_\infty) 
\cap \Ker(\bar{\mathfrak{f}}^\mathfrak{s}_\infty)$ then by Lemma \ref{lem:3.2} 
$x,y,z=0$ and $w \in \mathrm{A}_{\mathfrak{s}+1,-\mathfrak{s}}$. It follows that 
$\Ker(B_0) \cong \mathrm{E}_{1}$. For the 
second claim, note that $\Coker(B_1) 
\cong \Coker(\mathfrak{f}_\infty,\bar{\mathfrak{f}}_\infty)$. 
To see this, note that 
$\Coker(\mathfrak{f}_\infty,\bar{\mathfrak{f}}_\infty)$ is 
isomorphic to the intersection of the left kernel of the associated 
matrices and that 
\begin{align*}
\colvec{x & y}\colvec{ 0 & 0 \\ I & 0 }=0 \quad\text{and}\quad
\colvec{x & y}\colvec{A_1 & B_1\\ \ov{C}_1& D_1}\colvec{0&0\\ I&0}
\colvec{ A_0 & B_0 \\ C_0 & D_0}=0
\end{align*}
are true if and only if $y=0$ and $x B_1=0$. As we saw before 
\begin{equation}\label{eq:f-infty}
\begin{split}
\mathrm{Im}(\mathfrak{f}_\infty^\mathfrak{s})&= \mathrm{K}_{\mathfrak{s},*} \oplus 
[\mathrm{K}_{*,-\mathfrak{s}-1}]^{-\mathfrak{s}} \oplus (\bigoplus_{  
p \le \mathfrak{s} <-q } \mathrm{A}_{p,q}) \oplus 
(\bigoplus_{  p > \mathfrak{s} >-q} 
\mathrm{A}_{p,q})\quad\text{and}\\
\mathrm{Im}(\bar{\mathfrak{f}}_\infty^{\mathfrak{s}-1})
&= [\mathrm{K}_{\mathfrak{s}-1,*}]^{\mathfrak{s}} \oplus \mathrm{K}_{*,-\mathfrak{s}} \oplus 
(\bigoplus_{ p < \mathfrak{s} \le -q } \mathrm{A}_{p,q}) \oplus 
(\bigoplus_{ p > \mathfrak{s}>-q } 
\mathrm{A}_{p,q}) 
\end{split}
\end{equation}
Thus $\Coker(\mathfrak{f}^\mathfrak{s}_\infty,\bar{\mathfrak{f}}^\mathfrak{s}_\infty)
=\mathrm{A}_{\mathfrak{s},-\mathfrak{s}}$ and
$\Coker(B_1) \cong  \mathrm{E}_{0}$.
Similarly, $\Ker(B_1) \cong \Ker(\mathfrak{f}_0) \cap 
\Ker(\bar{\mathfrak{f}}_0) =\mathrm{Im}(\mathfrak{f}_\infty) 
\cap \mathrm{Im}(\bar{\mathfrak{f}}_\infty)$. 
It follows from Equation~\ref{eq:f-infty} that
\begin{equation}\label{eq:f-infty-intersection}
\mathrm{Im}(\mathfrak{f}_\infty^\mathfrak{s}) \cap 
\mathrm{Im}(\bar{\mathfrak{f}}_\infty^{\mathfrak{s}-1})= 
[\mathrm{K}_{\mathfrak{s}-1,*}]^s \oplus [\mathrm{K}_{*,-\mathfrak{s}-1}]^{-\mathfrak{s}} \oplus 
(\bigoplus_{ p < \mathfrak{s}  < -q } 
\mathrm{A}_{p,q}) \oplus (\bigoplus_{p > \mathfrak{s}>-q} \mathrm{A}_{p,q}), 
\end{equation}
proving the third isomorphism. 
Finally,
$\Coker(B_0)\cong\Coker(\mathfrak{f}_1,\bar{\mathfrak{f}}_1)=\frac 
{\Hbb_0}{\Ker(\mathfrak{f}_\infty)+\Ker(\bar{\mathfrak{f}}_\infty)}$. 
By Lemma~\ref{lem:3.2} 
\begin{align*}
\Ker(\mathfrak{f}_\infty^\mathfrak{s})= [\mathrm{K}_{*,-\mathfrak{s}-1}]_{-\mathfrak{s}} \oplus 
\bigoplus_{p > \mathfrak{s} =-q } \mathrm{A}_{p,q}\quad\text{and}\quad
\Ker(\bar{\mathfrak{f}}_\infty^\mathfrak{s})= [\mathrm{K}_{\mathfrak{s},*}]_{\mathfrak{s}+1} \oplus 
\bigoplus_{ p-1=\mathfrak{s}\ge -q} \mathrm{A}_{p,q}.
\end{align*}
The fourth isomorphism in the statement of the lemma 
follows immediately.
\end{proof}

\begin{lem}\label{lem:3.7}
With the above notation fixed, we have
\begin{align*}
\Ker(B_1 B_0) &\cong (\bigoplus_\mathfrak{s} [\mathrm{K}_{*,\mathfrak{s}}]_{\mathfrak{s}+1} \cap 
[\mathrm{K}_{*,\mathfrak{s}-1}]^{\mathfrak{s}}) \oplus (\bigoplus_{\mathfrak{s}} 
\mathrm{E}_{\mathfrak{s}}^{\max\{0,\mathfrak{s}\}})\quad\text{and}\\
\Coker(B_1 B_0) &\cong \bigoplus_\mathfrak{s} \frac{\mathrm{K}_{*,\mathfrak{s}}}
{[\mathrm{K}_{*,\mathfrak{s}-1}]^{\mathfrak{s}}+[\mathrm{K}_{*,\mathfrak{s}}]_{\mathfrak{s}+1}} \oplus 
(\bigoplus_{\mathfrak{s}} \mathrm{E}_{\mathfrak{s}}^{\mathrm{max}\{ 0,1-\mathfrak{s}\}}).
\end{align*}
\end{lem}
\begin{proof}
Note that $\Ker(B_1 B_0) \cong \Ker(\mathfrak{f}_0 
\bar{\mathfrak{f}}_\infty) \cap \Ker(\mathfrak{f}_\infty)$ 
and $\Ker(\mathfrak{f}_0 \bar{\mathfrak{f}}_\infty)=\bar{\mathfrak{f}}_\infty^{-1}
\Big(\text{Im}(\bar{\mathfrak{f}}_\infty)\cap
\text{Im}(\mathfrak{f}_\infty)\Big)$.
It then follows from Equation~\ref{eq:f-infty-intersection} that
\begin{align*}
(\bar{\mathfrak{f}}_\infty^{\mathfrak{s}+1})^{-1}\Big(\text{Im}
(\bar{\mathfrak{f}}_\infty^{\mathfrak{s}+1})\cap\Ker(\mathfrak{f}_0^{\mathfrak{s}+1})\Big)
= \mathrm{K}_{\mathfrak{s},*} \oplus [\mathrm{K}_{*,-\mathfrak{s}-2}]^{-\mathfrak{s}-1} \oplus 
\Big(\bigoplus_{p \le \mathfrak{s} <  -q-1} \mathrm{A}_{p,q}\Big) \oplus 
\Big(\bigoplus_{p \ge \mathfrak{s}+1\geq -q} \mathrm{A}_{p,q}\Big).
\end{align*}
Using Lemma~\ref{lem:3.2} and the above computation we find
\begin{align*}
\Ker(\mathfrak{f}_0^{\mathfrak{s}+1} \bar{\mathfrak{f}}_\infty^{\mathfrak{s}+1}) \cap 
\Ker(\mathfrak{f}_\infty^{\mathfrak{s}})= [\mathrm{K}_{*,-\mathfrak{s}-1}]_{-\mathfrak{s}} \cap 
[\mathrm{K}_{*,-\mathfrak{s}-2}]^{-\mathfrak{s}-1}\oplus 
\bigoplus_{p>\mathfrak{s} =-q }\mathrm{A}_{p,q}.
\end{align*}
This completes the proof of the first isomorphism in the statement 
of the lemma. For the second isomorphism, note that
$\Coker(B_1 B_0) \cong \Coker(\mathfrak{f}_\infty,
\bar{\mathfrak{f}}_\infty \mathfrak{f}_1)=\frac{\mathbb{H}_1}
{\mathrm{Im}(\mathfrak{f}_\infty)
+\mathrm{Im}(\bar{\mathfrak{f}}_\infty \mathfrak{f}_1)}$. 
We can also compute 
\begin{align*}
\mathrm{Im}(\mathfrak{f}_\infty^\mathfrak{s})&= \mathrm{K}_{\mathfrak{s},*} \oplus 
[\mathrm{K}_{*,-\mathfrak{s}-1}]^{-\mathfrak{s}} \oplus (\bigoplus_{  
p \le \mathfrak{s} <-q} \mathrm{A}_{p,q}) \oplus (\bigoplus_{ p > \mathfrak{s} >-q} 
\mathrm{A}_{p,q})
\quad\text{and}\quad\\
\text{Im}(\bar{\mathfrak{f}}_\infty^\mathfrak{s} \mathfrak{f}_1^{\mathfrak{s}-1}) 
&=\bar{\mathfrak{f}}_\infty^\mathfrak{s}(\text{Ker}(\mathfrak{f}_\infty^{\mathfrak{s}-1}))  
=\bar{\mathfrak{f}}_\infty^\mathfrak{s}\Big([\mathrm{K}_{*,-\mathfrak{s}}]_{-\mathfrak{s}+1} \oplus 
\bigoplus_{p>\mathfrak{s}-1=-q} 
\mathrm{A}_{p,q}\Big)=[\mathrm{K}_{*,-\mathfrak{s}}]_{-\mathfrak{s}+1} \oplus 
\bigoplus_{p > \mathfrak{s} =-q+1} \mathrm{A}_{p,q}\\
\Rightarrow\quad& 
\Coker(\mathfrak{f}_\infty^\mathfrak{s},\bar{\mathfrak{f}}_\infty^\mathfrak{s} 
\mathfrak{f}_1^{\mathfrak{s}-1})=\frac{\mathrm{K}_{*,-\mathfrak{s}}}
{[\mathrm{K}_{*,-\mathfrak{s}-1}]^{-\mathfrak{s}}+[\mathrm{K}_{*,-\mathfrak{s}}]_{-\mathfrak{s}+1}} \oplus 
(\bigoplus_{ p \le \mathfrak{s}=-q } \mathrm{A}_{p,q})
\end{align*}
This completes the proof of the lemma.
\end{proof}


\section{Proof of the main theorem}\label{sec:proof}
This section is devoted to completing the proof of 
Theorem~\ref{thm:main}.
\begin{proof}[of Theorem~\ref{thm:main}]
Let us assume that the inequality in 
Equation~\ref{eq:main-inequality} is not satisfied.
Proposition~\ref{prop:3} implies that it is enough to consider 
the special cases $\mathbf{S}_i$ for $i=1,\ldots,5$. 
On the other hand,
if $Y_2$ is a homology sphere $L$-space, it follows that 
$|a_0^2-a_1^2|=1$ and 
\begin{align*}
1&=y_\infty^2=a_1^2+a_0^2-2r_0^2-2r_1^2
+2(a_\infty^2-\delta_\infty^2)
\geq  a_1^2+a_0^2-2r_0^2-2r_1^2+2\max\{r_0^2,r_1^2\}.
\end{align*}
Combining the outcome of Proposition~\ref{prop:3} and the above 
observation, it follows that if the inequality in 
Equation~\ref{eq:main-inequality} is not satisfied, then 
$r_0^2=r_1^2=r_\infty^2=a_1^2=a_0^1-1:=a\leq a_\infty^2$. 
Let us first assume that $a_\infty^2=a_1^2=a$. This implies that 
\begin{equation*}
\Ker(B_0^2),\Coker(B_0^2),\Ker(B_1^2),\Coker(B_\infty^2)=0\quad
\text{and}\quad\Ker(B_\infty^2),\Coker(B_1^2)\neq 0.
\end{equation*}
After an admissible change of basis, we may assume that the matrices 
$\tau_\bullet^2$  are of the form 
\begin{equation*}
\renewcommand{\arraystretch}{0.8}%
{\scriptsize
\tau_1^2=\left(
\begin{array}{c@{\hspace{2bp}}c@{\hspace{2bp}};{2pt/2pt}c}
0 & 0 & I \\ 
0 & J & 0 \\ \hdashline[2pt/2pt]
I & 0 & 0 \\
\end{array}
\right)
}
\quad , \quad
\renewcommand{\arraystretch}{0.8}%
{\scriptsize
\tau_0^2=\left(
\begin{array}{c@{\hspace{2bp}};{2pt/2pt}c}
0 & I \\ \hdashline[2pt/2pt]
I & 0 \\ 
\end{array}
\right)
}
\quad , \quad
\renewcommand{\arraystretch}{0.8}%
{\scriptsize
\tau_\infty^2=\left(
\begin{array}{c@{\hspace{2bp}};{2pt/2pt}c@{\hspace{2bp}}c}
0 & X_1^2 & B \\ \hdashline[2pt/2pt]
* & x & y \\ 
* & x^\prime & y^\prime \\
\end{array}
\right).
}
\end{equation*}
After replacing the above block presentations,
we arrive at the following equivalent matrix
$$
\left( \begin{smallmatrix}
0 & B_1^1A_0^1 \otimes I & 0 & B_1^1B_0^1 \otimes I 
& 0 & 0& I \otimes I &0 \\
0 & 0 & B_1^1A_0^1 \otimes I & 0 & B_1^1B_0^1 \otimes I &0& 0 &0\\
I \otimes X_1^2 & 0 &  D_1^1A_0^1 \otimes BJ & 0 & 
D_1^1B_0^1 \otimes BJ & 0&B_0^1B_\infty^1 \otimes I &
B_0^1A_\infty^1\otimes I\\
I \otimes x & I \otimes I & D_1^1A_0^1 \otimes yJ & 0 & 
D_1^1B_0^1 \otimes yJ & 0& 0 & 0\\
I \otimes x^\prime & 0 & I+D_1^1A_0^1 \otimes y^\prime J & 0 & 
D_1^1B_0^1 \otimes y^\prime J & 0& 0&0  &0\\
B_\infty^1B_1^1 \otimes I & 0 & 0 & I \otimes X_1^2 & 
I \otimes B & B_\infty^1A_1^1\otimes I&
X_\infty^1B_\infty^1 \otimes X_1^2 &X_1^1A_\infty^1\otimes X_1^1\\
0&0&0&0&0&I\otimes I&0&0\\
0&0&0&0&0&0&0&I\otimes I
\end{smallmatrix} \right).
$$
It follows that $\text{Ker}(\mathfrak{D})$ has a 
subspace isomorphic to 
$\text{Ker}(B_1^1B_0^1) \otimes \text{Ker}(X_1^2)$ corresponding
to the fourth column in the above presentation.  Since 
$\Ker(X_1^2)\neq 0$ by \cite[Lemma 4.6]{Eftekhary-incompressible}, it 
follows from Lemma~\ref{lem:3.7} that the rank of $\Ker(\Dfrak)$
is greater than or equal to $\sum_{\mathfrak{s}>0}e_\mathfrak{s}$, 
where $e_\mathfrak{s}=e_\mathfrak{s}(K_1)$. 
After an admissible change of basis we may assume that
\begin{equation*}
\renewcommand{\arraystretch}{0.8}%
{\scriptsize
\tau_1^2=\left(
\begin{array}{c@{\hspace{2bp}}c@{\hspace{2bp}};{2pt/2pt}c}
a & b & A \\ 
c & d & X_\infty^2 \\ \hdashline[2pt/2pt]
* & * & 0 \\
\end{array}
\right)
}
\quad , \quad
\renewcommand{\arraystretch}{0.8}%
{\scriptsize
\tau_0^2=\left(
\begin{array}{c@{\hspace{2bp}};{2pt/2pt}c}
0 & I \\ \hdashline[2pt/2pt]
I & 0 \\ 
\end{array}
\right)
}
\quad , \quad
\renewcommand{\arraystretch}{0.8}%
{\scriptsize
\tau_\infty^2=\left(
\begin{array}{c@{\hspace{2bp}};{2pt/2pt}c@{\hspace{2bp}}c}
0 & 0 & I \\ \hdashline[2pt/2pt]
0 & J & 0 \\ 
I & 0 & 0 \\
\end{array}
\right).
}
\end{equation*}
Similar to the argument above, we can identify a subspace of the
 $\text{Coker}(\mathfrak{D})$ which is isomorphic to 
$\text{Coker}(B_1^1B_0^1) \otimes \text{Coker}(X_\infty^2)$ and
corresponds to the second row in the block presentation
(in this case, one needs to pass to an 
equivalent matrix using a cancellation). Since 
$\Coker(X_\infty^2)$ is non-trivial by 
\cite[Lemma 4.6]{Eftekhary-incompressible}, 
it follows from Lemma~\ref{lem:3.7} that the rank of 
$\Coker(\mathfrak{D})$ is greater than or equal to 
$\sum_{\mathfrak{s}\leq 0}e_\mathfrak{s}$. Together with the inequality we obtained for 
the rank of $\Ker(\Dfrak)$, this completes the proof of the 
theorem (when $a_\infty^2=a_1^2$).
Thus $a_\infty^2>a_1^2=a$, and it follows that
\begin{equation*}
\Ker(B_0^2),\Coker(B_\infty^2)=0\quad\text{and}\quad
\Ker(B_1^2),\Ker(B_\infty^2),\Coker(B_0^2),\Coker(B_1^2)\neq 0.
\end{equation*}
If $B_\infty^2 B_1^2$ is not surjective then part $(2-b)$ of 
Lemma \ref{lem:19} and Remark~\ref{remark} imply that 
\begin{align*}
&\rank(\Coker(\mathfrak{D}))\geq 
\rank\left(\Coker(B_1^1 B_0^1)\otimes\Coker(B_\infty^2 B_1^2)\right)
\geq \sum_{\mathfrak{s}\leq 0} e_\mathfrak{s},\\
&\rank (\text{Ker}(\mathfrak{D}))\geq 
\rank \left(\text{Ker}(B_1^1) \otimes \text{Ker}(B_1^2)\right)+ 
\rank\left(\text{Ker}(B_0^1) \otimes \text{Ker}(B_\infty^2)\right)
\geq e_1+\sum_{|\mathfrak{s}|>1}e_\mathfrak{s}.
\end{align*}
The last inequality in each case is a consequence of 
Lemma~\ref{lem:3.3} and Lemma~\ref{lem:3.7}.
These two inequalities combine to complete the proof of the 
theorem (under the assumption
that $B_\infty^2 B_1^2$ is not surjective). We may thus assume 
that  $B_\infty^2 B_1^2$ is surjective.\\

Since $a_1^2=a,a_0^2=a+1$ and $a_\infty^2=a+1+b$ for positive integers
$a,b$, one may consider the corresponding $4\times 4$, 
$5\times 5$ and $3\times 3$ block presentations of the form
\begin{equation*}
\renewcommand{\arraystretch}{0.8}%
{\scriptsize
\tau_1=\left(
\begin{array}{c@{\hspace{2bp}}c@{\hspace{2bp}};{2pt/2pt}c@{\hspace{2bp}}c@{\hspace{2bp}}c}
* & * & M & N & P \\ 
0 & 0 & M^\prime & N^\prime & P^\prime \\ \hdashline[2pt/2pt]
* & * & x & y & z \\
* & * & x^\prime & y^\prime & z^\prime \\
* & * & x^{\prime \prime} & y^{\prime \prime} & z^{\prime \prime} \\
\end{array}
\right)
}
\quad , \quad
\renewcommand{\arraystretch}{0.8}%
{\scriptsize
\tau_0=\left(
\begin{array}{c@{\hspace{2bp}}c@{\hspace{2bp}}c
@{\hspace{2bp}};{2pt/2pt}c}
* & * & * & E \\ 
* & * & * & F \\ 
* & * & * & G \\ \hdashline[2pt/2pt]
* & * & * & 0 
\end{array}
\right)
}
\quad , \quad
\renewcommand{\arraystretch}{0.8}%
{\scriptsize
\tau_\infty=\left(
\begin{array}{c@{\hspace{2bp}}
c@{\hspace{2bp}};{2pt/2pt}c@{\hspace{2bp}}c}
0 & 0 & 0 & I \\\hdashline[2pt/2pt]
* & * & * & 0 \\
* & * & * & 0
\end{array}
\right),
}
\end{equation*}
such that $B_\infty^2 B_1^2=\colvec{0&0&I}$. It follows that 
$M^\prime, N^\prime=0$ and $P^\prime=I$. 
On the other hand, it follows from  $A_1^2B_1^2+B_1^2D_1^2=0$ 
that $x^{\prime \prime}, y^{\prime \prime},z^{\prime \prime}=0$. 
After replacing the above block forms the matrix 
$\mathfrak{D}=\Dfrak(K_1,K_2)$ takes the form
$$
\left( \begin{smallmatrix}
\star & \star & \star & \star & \star & \star & \star & \star 
& \star & \star & \star  & \star\\
\star & \star & \star & \star & \star & \star & \star & \star 
& \star & \star & \star  &\star\\
0 & 0 & I \otimes I & 0 & 0 & 0 & 0 & 0 & 0 & 0 & 
B_0^1B_\infty^1 \otimes I  & \star\\
\star & \star & \star & \star & \star & \star & \star & 
\star & \star & \star & \star  &\star\\
\star & \star & \star & \star & \star & \star & \star & 
\star & \star & \star & \star  &\star\\
\star & \star & \star & \star & \star & \star & \star & 
\star & \star & \star & \star  &\star\\
\star & \star & \star & \star & \star & \star & \star & 
\star & \star & \star & \star  &\star\\
0 & 0 & B^1_\infty B^1_1 \otimes I & 0 & 0 & 0 & I \otimes X_1^2 
& 0 & 0 & B^1_\infty A^1_1 \otimes I & 0 &\star\\
\star & \star & \star & \star & \star & \star & \star 
& \star & \star & \star & \star  &\star\\
\star & \star & \star & \star & \star & \star & \star & 
\star & \star & \star & \star  &\star\\
0 & 0 & 0 & 0 & 0 & 0 & 0 & 0 & 0 & I \otimes I & 0  &\star\\
0&0&0&0&0&0&0&0&0&0&0&I\otimes I
\end{smallmatrix} \right)
$$
We may then cancel the tenth and twelfth columns against the 
eleventh and twelfth rows. Moreover, we  may
 add $B^1_\infty B^1_1 \otimes I$ times the third row to the 
eighth row  to arrive at a new equivalent matrix. 
Considering the eighth row of the aforementioned matrix,
we observe that $\text{Coker}(\mathfrak{D})$ contains a 
subspace isomorphic to 
$\text{Coker}(X_1^1 B_\infty^1) \otimes \text{Coker}(X_1^2)$,
independent from the previously identified subspace of 
$\Coker(\Dfrak)$. \\

Part $(2-a)$ of Lemma \ref{lem:19} implies that there is a
subspace of $\Ker(\Dfrak)$ which is isomorphic to 
$\Ker(B_\infty^1 B_1^1)\otimes \Ker(B_1^2B_0^2)$. 
Let us first assume that $B_1^2B_0^2$ is not 
injective. Since $\Ker(B_1^2)$ is at least $2$ dimensional, 
it follows from Remark~\ref{remark} that 
\begin{align*}
&y_\infty^1>\rank\left(\Ker(B_1^1)\oplus
\Ker(B_0^1)\oplus\Ker(B_\infty^1B_1^1)\oplus
\Coker(B_1^1)\oplus\Coker(X_1^1B_\infty^1)\right),
\quad \text{while}\\
&\rank(\Ker(B_\infty^1B_1^1))=a_\infty^1-\rank(B_\infty^1B_1^1)
\geq a_\infty^1-\min\{r_\infty^1,r_1^1\}\quad\text{and}\\
&\rank(\Coker(X_1^1B_\infty^1))=a_1^1-\rank(X_1^1B_\infty^1)
\geq a_1^1-\min\{r_\infty^1,r_1^1, r_0^1\}.
\end{align*}
The above three inequalities imply
\begin{align*}
2a_\infty^1+a_0^1+a_1^1-2r_0^1-2r_1^1-2\delta_\infty^1&>
2a_\infty^1+2a_1^1+a_0^1-2r_1^1-r_0^1-\min\{r_1^1,r_\infty^1\}-
\min\{r_0^1,r_1^1,r_\infty^1\}.
\end{align*}
This latter inequality simplifies to give
\begin{align*}
\min\{r_1^1,r_\infty^1\}+\min\{r_0^1,r_1^1,r_\infty^1\}&>
a_1^1+r_0^1.
\end{align*}
Since $\min\{r_1^1,r_\infty^1\}
\leq r_\infty^1\leq a_1^1$ 
and $\min\{r_0^1,r_1^1,r_\infty^1\}\leq r_0^1$, the above 
conclusion leads to a contradiction. In particular, it follows that 
$B_1^2B_0^2$ is injective. An argument similar to the one used in the
previous paragraphs then implies that $\Ker(\Dfrak)$ includes a 
subspace isomorphic to $\Ker(B_0^1X_1^1)\otimes \Ker(X_1^2)$,
which intersects trivially the previously identified subspaces of 
$\Ker(\Dfrak)$. In particular, since $\Ker(X_1^2)$ is non-trivial,
the rank $y$ of $\widehat{\HF}(Y(K_1,K_2))$ is not less than the rank
of
\begin{align*}
\Ker(B_1^1)\oplus\Ker(B_0^1)\oplus\Ker(B_0^1X_1^1)\oplus
\Coker(B_1^1)\oplus\Coker(X_1^1B_\infty^1).
\end{align*}
Since $(X_1^1)^2=0$, it follows that the rank of 
$\Ker(B_0^1X_1^1)\oplus\Coker(X_1^1B_\infty^1)$ is at least 
equal to $(a_1^1/2)+(a_1^1/2)=a_1^1$.  We thus conclude
\begin{align*}
a_1^1\geq \sum_{\mathfrak{s}>0} e_\mathfrak{s},\quad\rank(\Ker(B_1^1))\geq 
\sum_{\mathfrak{s}\neq -1,0,1} e_\mathfrak{s},\quad
\rank(\Ker(B_0^1))=e_1\quad\text{and}\quad\rank(\Coker(B_1^1))=e_0.
\end{align*}
These equalities and inequalities imply that
\begin{align*}
y&\geq  \sum_{\mathfrak{s}>0} e_0+e_1+
e_\mathfrak{s}+\sum_{\mathfrak{s}\neq -1,0,1} e_\mathfrak{s}
\geq -e_{-1}+\sum_\mathfrak{s} e_\mathfrak{s}
+\sum_{\mathfrak{s}>0}e_\mathfrak{s}
\geq y_\infty^1-e_{-1}+e_1.
\end{align*}
The above observation finishes the proof if $e_1\geq e_{-1}$.
 We are thus lead to further assume that $e_1(K_1)<e_{-1}(K_1)$.
Let $\ov{K_i}\subset \ov{Y_i}$ denote the mirror of $K_i$, where 
$\ov{Y_i}$ is the manifold $Y_i$ with the opposite orientation. Since 
\[\widehat{\HF}(Y(K_1,K_2))\simeq \widehat{\HF}(\ov{Y(K_1,K_2)})
=\widehat{\HF}(Y(\ov{K_1},\ov{K_2}),\]
it follows that we can replace $K_i$ with $\ov{K_i}$ for 
$i=1,2$ (simultaneously).  The knot Floer complex associated 
with $\ov{K_i}$ is the dual of 
the chain complex associated with $K_i$. It follows that 
$\mathrm{E}_\mathfrak{s}(\ov{K_i})$ is isomorphic to 
$\mathrm{E}_{-\mathfrak{s}}({K_i})$. In particular, 
$e_1(\ov{K_1})>e_{-1}(\ov{K_1})$, and the proof is complete.
\end{proof}
Theorem~\ref{thm:main} may be compared with the result of the second 
author \cite[Theorem 1.1]{Eftekhary-incompressible} 
(c.f. \cite[Theorem 1]{Hedden-Levine}) that splicing the complement
of non-trivial knots inside homology spheres never produces an
$L$-space.

 \bibliographystyle{plain}

\end{document}